\documentclass[11pt]{amsart}
\usepackage[utf8]{inputenc}
\usepackage{amsmath}
\usepackage{amsfonts}
\usepackage{amssymb}
\usepackage{tikz}
\usepackage{amsmath}
\usepackage{xparse}
\usetikzlibrary{calc,angles,positioning,intersections,quotes,decorations.markings}
\usetikzlibrary{arrows}

\newcommand{\eqarrow}[2][$\leftrightarrow$]{
    \begin{tikzpicture}
        \begin{scope}[shift={(1,0)}]
            #2
        \end{scope}
        \node at (0,1)(a){#1};
    \end{tikzpicture}
}

\newcommand{\unit}{
\node[mult] at (1,1) (u){};
\node at (1,0) (a){};
\draw[common] (u) to (a);
}

\newcommand{\counit}{
\node at (1,1) (u){};
\node[mult] at (1,0) (a){};
\draw[common] (u) to (a);
}

\newcommand{\pairing}{
	\node [mult] at (1,1) (m){};
    \node at (0,2) (a){};
    \node at (2,2) (b){};
    \node[mult] at (1,0) (c){};
    \draw[common] (a) to (m);
    \draw[common](b) to (m);
    \draw[common](m) -- (c);
}

\newcommand{\copairing}{
	\node [delt] at (1,1)(d){};
	
	\node[delt] at (1,2)(a){};
	\node at (0,0)(b){};
	\node at (2,0)(c){};
	\draw[common](a) to [out=270,in=90](d);
	\draw[common](d) to [out=205,in=60](b);
	\draw[common](d) to [out=-25,in=120](c);
}

\newcommand{\qcomwcounit}{
    \node at (0,3) (a){};
    \node at (2,3)(b){};
    \node[mult] at (1,0)(c){};
    \node at (3,0)(f){};
    \node [mult] at (1,1) (m){};
    \node [delt] at (2,2) (d){};
    \draw[common](a)[out=270,in=135] to (m);
    \draw[common](m) to (c);
    \draw[common](b) to (d);
    \draw[common](d) to (m);
    \draw[common](d)[out=-45,in=90] to (f);
}

\newcommand{\ncomwcounit}{
    \node at (0,3)(a){};
    \node at (2,3)(b){};
    \node[mult] at (0,0)(c){};
    \node at (2,0)(f){};
    \node [mult] at (1,2)(m){};
    \node [delt] at (1,1)(d){};
    \draw[common](a) to (m);
    \draw[common](b) to (m);
    \draw[common](m) to (d);
    \draw[common](d) to (c);
    \draw[common](d) to (f);
}

\newcommand{\multi}{
	\node [mult] at (1,1) (m){};
    \node at (0,2) (a){};
    \node at (2,2) (b){};
    \node at (1,0) (c){};
    \draw[common] (a) to (m);
    \draw[common](b) to (m);
    \draw[common](m) -- (c);
}
\newcommand{\bcom}{
    \node at (3,6) (a){};
	\node at (4,6) (b){};
    \node at (4,0) (c){};
    \node[mult] at (2,4) (upm){};
	\node[mult] at (3,2) (lom){};
	\node[delt] at (3,1) (d){};
	\draw[common] (b)[out=270,in=45] to (lom);
	\draw[common] (lom) to (d);
	\draw[common] (d) to (c);
	
	\draw[common] (a)[out=270,in=45] to (upm);
	\draw[common] (upm) to ($(upm)+(0,-1)$) to (lom);
	\draw[common] (d)[out=225,in=0] to ($(d) - (1,.5)$)[out=180,in=225]
	to (1,5)[out = 45,in=0] to (1,5.5)[out=180,in=135] to (1,5)[out=-45] to (upm);
}
\newcommand{\ccom}{
    \node at (1,4) (b){};
    \node at (2.5,4) (c){};
    \node at (1,0) (a){};
    \node [delt] at (1,3) (d){};
    \node [mult] at (2,2) (upm){};
    \node [mult] at (1,1) (lom){};
    \draw[common] (b) to (d);
    \draw[common] (d) to (upm);
	\draw[common] (d)[out=225,in=135] to (lom);
	\draw[common] (lom) to (a);
	\draw[common] (c)[out=270,in=45] to (upm);
	\draw[common] (upm) to (lom);
}
\newcommand{\dcom}{
	\node at (2,4) (c){};
	\node at (0,4) (b){};
	\node at (2,0) (a){};
	\node[delt] at (2,3) (d){};
	\node[mult] at (1,2) (upm){};
	\node[mult] at (2,1) (lom){};
	\draw[common](b)[out=270,in=135] to (upm);
	\draw[common](upm) to (lom);
	\draw[common](lom) to (a);
	\draw[common](c) to (d);
	\draw[common](d) to (upm);
	\draw[common](d)[out=-45,in=45] to (lom);
}
\newcommand{\edelt}{
	\node [delt] at (1,1)(d){};
	
	\node at (1,2)(a){};
	\node at (0,0)(b){};
	\node at (2,0)(c){};
	\draw[common](a) to [out=270,in=90](d);
	\draw[common](d) to [out=205,in=60](b);
	\draw[common](d) to [out=-25,in=120](c);
}
\newcommand{\fcom}{
    \node at (4,5) (b){};
	\node at (4,0) (c){};
	\node at (3,0) (a){};
    \node[mult] at (3,3) (m){};
	\node[delt] at (2,1) (lod){};
	\node[delt] at (3,2) (upd){};
	\draw[common] (b)[out=270,in=45] to (m);
	\draw[common] (m) to (upd);
	\draw[common] (upd) to ($(upd) - (1,.5)$) to (lod);
	\draw[common] (lod)[out=-45,in=90] to (a);
	\draw[common] (upd)[out=-45,in=90] to (c);
	\draw[common] (lod)[out=225,in=0] to ($(lod) - (1,.5)$)[out=180,in=225] to 
	(1,4)[out=45,in=0] to (1,4.5)[out=180,in=225] to (1,4)[out=-45,in=135] to (m);
}

\newcommand{\gcom}{
    \node at (1,4) (b){};
    \node at (1,0) (a){};
    \node at (3,0) (c){};
    \node [delt] at (1,3) (upd){};
    \node [delt] at (2,2) (lod){};
    \node [mult] at (1,1) (m){};
    \draw[common](b) -- (upd);
    \draw[common](upd)[out=225,in=135] to (m);
    \draw[common](m) -- (a);
    
    \draw[common](upd) to (lod);
    \draw[common](lod) to (m);
    \draw[common](lod)[out=-45,in=90] to (c);
}

\newcommand{\hcom}{
    \node at (2,4) (b){};
    \node at (0,0) (a){};
    \node at (2,0) (c){};
    \node [delt] at (2,3) (upd){};
    \node [delt] at (1,2) (lod){};
    \node [mult] at (2,1) (m){};
    
    \draw[common](b) -- (upd);
    \draw[common](upd) -- (lod);
    \draw[common](lod)[out=225,in=90] to (a);
    \draw[common](upd)[out=-45,in=45] to (m);
    \draw[common](lod) -- (m);
    \draw[common](m) to (c);
}
\newcommand{\icom}{
	\node at (0,3) (b){};
	\node at (2,3) (c){};
	\node at (4,3) (f){};
	\node at (2,0) (a){};
	\node[mult] at (1,2) (upm){};
	\node[mult] at (2,1) (lom){};
	\draw[common] (b) to (upm);
	\draw[common] (c) to (upm);
	\draw[common] (upm)to (lom);
	\draw[common] (lom) to (a);
	\draw[common] (f) to (lom);
}
\newcommand{\jcom}{
	\node at (0,3) (b){};
	\node at (2,3) (c){};
	\node at (4,3) (f){};
	\node at (2,0) (a){};
	\node[mult] at (3,2) (upm){};
	\node[mult] at (2,1) (lom){};
	\draw[common] (b) to (lom);
	\draw[common] (c) to (upm);
	\draw[common] (upm)to (lom);
	\draw[common] (lom) to (a);
	\draw[common] (f) to (upm);
}
\newcommand{\kcom}{
	\node at (2,3) (c){};
	\node at (0,0) (b){};
	\node at (2,0) (a){};
	\node at (4,0) (f){};
	\node [delt] at (1,1)(lod){};
	\node [delt] at (2,2)(upd){};
	\draw[common] (lod) to (b);
	\draw[common] (lod) to (a);
	\draw[common] (c) to (upd);
	\draw[common] (upd)[out=225,in=90] to (lod);
	\draw[common] (upd)[out=-45,in=90] to (f);
}
\newcommand{\lcom}{
	\node at (2,3) (c){};
	\node at (0,0) (b){};
	\node at (2,0) (a){};
	\node at (4,0) (f){};
	\node [delt] at (3,1)(lod){};
	\node [delt] at (2,2)(upd){};
	\draw[common] (upd)[out=225,in=90] to (b);
	\draw[common] (c) to (upd);
	\draw[common] (upd)[out=-45,in=90] to (lod);
	\draw[common] (lod)[out=225,in=90] to (a);
	\draw[common] (lod)[out=-45,in=90] to (f);
}
\newcommand{\mcom}{
	\node at (1,3) (c){};
	\node at (3,3) (f){};   
	\node at (0,0) (b){};
	\node at (2,0) (a){};
	\node [mult] at (2,1) (m){};
	\node [delt] at (1,2) (d){};
	\draw[common](c) to (d);
	\draw[common](d) to (m);
	\draw[common](m) to (a);
	
	\draw[common](d)[out=225,in=90] to (b);
	\draw[common](f)[out=270,in=45] to (m);
}
\newcommand{\ncom}{
    \node at (0,3)(a){};
    \node at (2,3)(b){};
    \node at (0,0)(c){};
    \node at (2,0)(f){};
    \node [mult] at (1,2)(m){};
    \node [delt] at (1,1)(d){};
    \draw[common](a) to (m);
    \draw[common](b) to (m);
    \draw[common](m) to (d);
    \draw[common](d) to (c);
    \draw[common](d) to (f);
}
\newcommand{\ocom}{
    \node at (1,4)(f){};
    \node at (2,4)(b){};
    \node at (1,0)(a){};
    \node at (2,0)(c){};
	\node [delt] at (1,3)(d){};
    \node [mult] at (1,1)(m){};
    
    \draw[common](b) to [out=270,in=30](m);
    \draw[common](f) to [out=270,in=90](d);
    \draw[common](d) to [out=200,in=135](m);
    
  \draw[common](m) to [out=270,in=90](a);
  \draw[common](d) to [out=-20, in=90](c);
}
\newcommand{\pcom}{
    \node at (3,4) (a){};
    \node at (1,4) (c){};
    \node at (0,0) (b){};
    \node at (2,0) (f){};
    \node[delt] at (1,3) (d){};
    \node[mult] at (2,1) (m){};
    \draw[common] (c) to (d);
    \draw[common] (d)[out=225,in=90] to (b);
    \draw[common] (d)[out=-45,in=45] to (m);
    \draw[common] (m) to (f);
    \draw[common] (a)[out=270,in=135] to (m);
}

\newcommand{\qcom}{
    \node at (0,3) (a){};
    \node at (2,3)(b){};
    \node at (1,0)(c){};
    \node at (3,0)(f){};
    \node [mult] at (1,1) (m){};
    \node [delt] at (2,2) (d){};
    \draw[common](a)[out=270,in=135] to (m);
    \draw[common](m) to (c);
    \draw[common](b) to (d);
    \draw[common](d) to (m);
    \draw[common](d)[out=-45,in=90] to (f);
}
\newcommand{\rcom}{
	\node at (0,4) (a){};
	\node at (2,4) (c){};
	\node at (1,0) (f){};
	\node at (3,0) (b){};
	\node[delt] at (2,3) (d){};
	\node[mult] at (1,1) (m){};
	\draw[common] (a)[out=270,in=135] to (m);
	\draw[common] (c) to (d);
	\draw[common] (m) to (f);
	\draw[common] (d)[out=-45,in=45] to (m);
	\draw[common] (d)[out=225,in=90] to (b);
}
\newcommand{\scom}{
    \node at (1,6) (c){};
    \node at (0,4) (b){};
    \node at (0,1) (a){};
    \node[mult] at (2,5) (upm){};
	\node[mult] at (1,3) (lom){};
	\node[delt] at (1,2) (d){};
	\draw[common] (b) to (lom);
	\draw[common] (lom) to (d);
	\draw[common] (d) to (a);
	\draw[common] (c) to (upm);
	\draw[common] (upm) to (2,4) to (lom);
	\draw[common](d) to (2.75,1)[out=-45,in=0] to (2.75,.5)[out=180,in=225] to (2.75,1.5)
	[out=45,in=45] to (upm);
}
\newcommand{\tcom}{
    \node at (0,5) (b){};
    \node at (0,2) (a){};
    \node at (1,0) (c){};
    \node [mult] at (1,4) (m){};
    \node [delt] at (1,3) (upd){};
    \node [delt] at (2,2) (lod){};
    \draw[common](b) to (m);
    \draw[common](m) to (upd);
    \draw[common](upd) to (a);
    \draw[common](upd) to (lod);
    \draw[common](lod) to (c);
	
	\draw[common](lod) to (2.75,1)[out=-45,in=0] to (2.75,.5)[out=180,in=225] to (2.75,1.5)
	[out=45,in=45] to (m);
}

\newcommand{\actred}{
	\node [mult] at (1,1) (m){};
    \node at (0,2) (a){};
    \node at (2,2) (b){};
    \node at (1,0) (c){};
    \draw[dpoint] (a) to (m);
    \draw[common](b) to (m);
    \draw[dpoint](m) -- (c);
}
\newcommand{\actdef}{
	\node [mult] at (1,1) (m){};
    \node at (0,2) (a){};
    \node at (2,2) (b){};
    \node at (1,0) (c){};
    \draw[defect] (a) to (m);
    \draw[dpoint](b) to (m);
    \draw[dpoint](m) -- (c);
}
\newcommand{\cred}{
    \node at (1,4) (b){};
    \node at (2.5,4) (c){};
    \node at (1,0) (a){};
    \node [delt] at (1,3) (d){};
    \node [mult] at (2,2) (upm){};
    \node [mult] at (1,1) (lom){};
    \draw[dpoint] (b) to (d);
    \draw[common] (d) to (upm);
	\draw[dpoint] (d)[out=225,in=135] to (lom);
	\draw[dpoint] (lom) to (a);
	\draw[common] (c)[out=270,in=45] to (upm);
	\draw[common] (upm) to (lom);
}
\newcommand{\ddef}{
	\node at (2,4) (c){};
	\node at (0,4) (b){};
	\node at (2,0) (a){};
	\node[delt] at (2,3) (d){};
	\node[mult] at (1,2) (upm){};
	\node[mult] at (2,1) (lom){};
	\draw[defect](b)[out=270,in=135] to (upm);
	\draw[dpoint](upm) to (lom);
	\draw[dpoint](lom) to (a);
	\draw[dpoint](c) to (d);
	\draw[dpoint](d) to (upm);
	\draw[common](d)[out=-45,in=45] to (lom);
}
\newcommand{\dred}{
	\node at (2,4) (c){};
	\node at (0,4) (b){};
	\node at (2,0) (a){};
	\node[delt] at (2,3) (d){};
	\node[mult] at (1,2) (upm){};
	\node[mult] at (2,1) (lom){};
	\draw[dpoint](b)[out=270,in=135] to (upm);
	\draw[dpoint](upm) to (lom);
	\draw[dpoint](lom) to (a);
	\draw[common](c) to (d);
	\draw[common](d) to (upm);
	\draw[common](d)[out=-45,in=45] to (lom);
}
\newcommand{\edefdelt}{
	\node [delt] at (1,1)(d){};
	
	\node at (1,2)(a){};
	\node at (0,0)(b){};
	\node at (2,0)(c){};
	\draw[dpoint](a) to [out=270,in=90](d);
	\draw[defect](d) to [out=205,in=60](b);
	\draw[dpoint](d) to [out=-25,in=120](c);
}
\newcommand{\eredelt}{
		\node [delt] at (1,1)(d){};
        
        \node at (1,2)(a){};
        \node at (0,0)(b){};
        \node at (2,0)(c){};
        \draw[dpoint](a) to [out=270,in=90](d);
        \draw[dpoint](d) to [out=205,in=60](b);
        \draw[common](d) to [out=-25,in=120](c);
}

\newcommand{\gred}{
    \node at (1,4) (b){};
    \node at (1,0) (a){};
    \node at (3,0) (c){};
    \node [delt] at (1,3) (upd){};
    \node [delt] at (2,2) (lod){};
    \node [mult] at (1,1) (m){};
    \draw[dpoint](b) -- (upd);
    \draw[dpoint](upd)[out=225,in=135] to (m);
    \draw[dpoint](m) -- (a);
    
    \draw[common](upd) to (lod);
    \draw[common](lod) to (m);
    \draw[common](lod)[out=-45,in=90] to (c);
}

\newcommand{\hdef}{
    \node at (2,4) (b){};
    \node at (0,0) (a){};
    \node at (2,0) (c){};
    \node [delt] at (2,3) (upd){};
    \node [delt] at (1,2) (lod){};
    \node [mult] at (2,1) (m){};
    
    \draw[dpoint](b) -- (upd);
    \draw[dpoint](upd) -- (lod);
    \draw[defect](lod)[out=225,in=90] to (a);
    \draw[common](upd)[out=-45,in=45] to (m);
    \draw[dpoint](lod) -- (m);
    \draw[dpoint](m) to (c);
}
\newcommand{\hred}{
    \node at (2,4) (b){};
    \node at (0,0) (a){};
    \node at (2,0) (c){};
    \node [delt] at (2,3) (upd){};
    \node [delt] at (1,2) (lod){};
    \node [mult] at (2,1) (m){};
    
    \draw[dpoint](b) -- (upd);
    \draw[dpoint](upd) -- (lod);
    \draw[dpoint](lod)[out=225,in=90] to (a);
    \draw[common](upd)[out=-45,in=45] to (m);
    \draw[common](lod) -- (m);
    \draw[common](m) to (c);
}

\newcommand{\idef}{
	\node at (0,3) (b){};
	\node at (2,3) (c){};
	\node at (4,3) (f){};
	\node at (2,0) (a){};
	\node[mult] at (1,2) (upm){};
	\node[mult] at (2,1) (lom){};
	\draw[defect] (b) to (upm);
	\draw[dpoint] (c) to (upm);
	\draw[dpoint] (upm)to (lom);
	\draw[dpoint] (lom) to (a);
	\draw[common] (f) to (lom);
}
\newcommand{\ired}{
	\node at (0,3) (b){};
	\node at (2,3) (c){};
	\node at (4,3) (f){};
	\node at (2,0) (a){};
	\node[mult] at (1,2) (upm){};
	\node[mult] at (2,1) (lom){};
	\draw[dpoint] (b) to (upm);
	\draw[common] (c) to (upm);
	\draw[dpoint] (upm) to (lom);
	\draw[dpoint] (lom) to (a);
	\draw[common] (f) to (lom);
}
\newcommand{\jdef}{
	\node at (0,3) (b){};
	\node at (2,3) (c){};
	\node at (4,3) (f){};
	\node at (2,0) (a){};
	\node[mult] at (3,2) (upm){};
	\node[mult] at (2,1) (lom){};
	\draw[defect] (b) to (lom);
	\draw[dpoint] (c) to (upm);
	\draw[dpoint] (upm)to (lom);
	\draw[dpoint] (lom) to (a);
	\draw[common] (f) to (upm);
}
\newcommand{\jred}{
	\node at (0,3) (b){};
	\node at (2,3) (c){};
	\node at (4,3) (f){};
	\node at (2,0) (a){};
	\node[mult] at (3,2) (upm){};
	\node[mult] at (2,1) (lom){};
	\draw[dpoint] (b) to (lom);
	\draw[common] (c) to (upm);
	\draw[common] (upm)[out=225,in=45] to (lom);
	\draw[dpoint] (lom) to (a);
	\draw[common] (f) to (upm);
}
\newcommand{\kdef}{
	\node at (2,3) (c){};
	\node at (0,0) (b){};
	\node at (2,0) (a){};
	\node at (4,0) (f){};
	\node [delt] at (1,1)(lod){};
	\node [delt] at (2,2)(upd){};
	\draw[defect] (lod) to (b);
	\draw[dpoint] (lod) to (a);
	\draw[dpoint] (c) to (upd);
	\draw[dpoint] (upd)[out=225,in=90] to (lod);
	\draw[common] (upd)[out=-45,in=90] to (f);
}
\newcommand{\kred}{
	\node at (2,3) (c){};
	\node at (0,0) (b){};
	\node at (2,0) (a){};
	\node at (4,0) (f){};
	\node [delt] at (1,1)(lod){};
	\node [delt] at (2,2)(upd){};
	\draw[dpoint] (lod) to (b);
	\draw[common] (lod) to (a);
	\draw[dpoint] (c) to (upd);
	\draw[dpoint] (upd)[out=225,in=90] to (lod);
	\draw[common] (upd)[out=-45,in=90] to (f);
}
\newcommand{\ldef}{
	\node at (2,3) (c){};
	\node at (0,0) (b){};
	\node at (2,0) (a){};
	\node at (4,0) (f){};
	\node [delt] at (3,1)(lod){};
	\node [delt] at (2,2)(upd){};
	\draw[defect] (upd)[out=225,in=90] to (b);
	\draw[dpoint] (c) to (upd);
	\draw[dpoint] (upd)[out=-45,in=90] to (lod);
	\draw[dpoint] (lod)[out=225,in=90] to (a);
	\draw[common] (lod)[out=-45,in=90] to (f);
}
\newcommand{\lred}{
	\node at (2,3) (c){};
	\node at (0,0) (b){};
	\node at (2,0) (a){};
	\node at (4,0) (f){};
	\node [delt] at (3,1)(lod){};
	\node [delt] at (2,2)(upd){};
	\draw[dpoint] (upd)[out=225,in=90] to (b);
	\draw[dpoint] (c) to (upd);
	\draw[common] (upd)[out=-45,in=90] to (lod);
	\draw[common] (lod)[out=225,in=90] to (a);
	\draw[common] (lod)[out=-45,in=90] to (f);
}
\newcommand{\mdef}{
	\node at (1,3) (c){};
	\node at (3,3) (f){};   
	\node at (0,0) (b){};
	\node at (2,0) (a){};
	\node [mult] at (2,1) (m){};
	\node [delt] at (1,2) (d){};
	\draw[dpoint](c) to (d);
	\draw[dpoint](d) to (m);
	\draw[dpoint](m) to (a);
	
	\draw[defect](d)[out=225,in=90] to (b);
	\draw[common](f)[out=270,in=45] to (m);
}
\newcommand{\mred}{
	\node at (1,3) (c){};
	\node at (3,3) (f){};   
	\node at (0,0) (b){};
	\node at (2,0) (a){};
	\node [mult] at (2,1) (m){};
	\node [delt] at (1,2) (d){};
	\draw[dpoint](c) to (d);
	\draw[dpoint](d)[out=225,in=90] to (b);
	
	\draw[common](d) to (m);
	\draw[common](m) to (a);
	\draw[common](f)[out=270,in=45] to (m);
}
\newcommand{\ndefa}{
    \node at (0,3)(a){};
    \node at (2,3)(b){};
    \node at (0,0)(c){};
    \node at (2,0)(f){};
    \node [mult] at (1,2)(m){};
    \node [delt] at (1,1)(d){};
    \draw[dpoint](a) to (m);
    \draw[common](b) to (m);
    \draw[dpoint](m) to (d);
    \draw[defect](d) to (c);
    \draw[dpoint](d) to (f);
}
\newcommand{\ndefb}{
    \node at (0,3)(a){};
    \node at (2,3)(b){};
    \node at (0,0)(c){};
    \node at (2,0)(f){};
    \node [mult] at (1,2)(m){};
    \node [delt] at (1,1)(d){};
    \draw[defect](a) to (m);
    \draw[dpoint](b) to (m);
    \draw[dpoint](m) to (d);
    \draw[dpoint](d) to (c);
    \draw[common](d) to (f);
}

\newcommand{\nred}{
    \node at (0,3)(a){};
    \node at (2,3)(b){};
    \node at (0,0)(c){};
    \node at (2,0)(f){};
    \node [mult] at (1,2)(m){};
    \node [delt] at (1,1)(d){};
    \draw[dpoint](a) to (m);
    \draw[common](b) to (m);
    \draw[dpoint](m) to (d);
    \draw[dpoint](d) to (c);
    \draw[common](d) to (f);
}

\newcommand{\odef}{
    \node at (1,4)(f){};
    \node at (2,4)(b){};
    \node at (1,0)(a){};
    \node at (2,0)(c){};
	\node [delt] at (1,3)(d){};
    \node [mult] at (1,1)(m){};
    
    \draw[dpoint](b) to [out=270,in=30](m);
    \draw[dpoint](f) to [out=270,in=90](d);
    \draw[defect](d) to [out=200,in=135](m);
    
  \draw[dpoint](m) to [out=270,in=90](a);
  \draw[dpoint](d) to [out=-20, in=90](c);
}
\newcommand{\ored}{
    \node at (1,4)(f){};
    \node at (2,4)(b){};
    \node at (1,0)(a){};
    \node at (2,0)(c){};
	\node [delt] at (1,3)(d){};
    \node [mult] at (1,1)(m){};
    
    \draw[common](b) to [out=270,in=30](m);
    \draw[dpoint](f) to [out=270,in=90](d);
    \draw[dpoint](d) to [out=200,in=135](m);
    
  \draw[dpoint](m) to [out=270,in=90](a);
  \draw[common](d) to [out=-20, in=90](c);
}
\newcommand{\preda}{
	\node at (3,4) (a){};
	\node at (1,4) (c){};
	\node at (0,0) (b){};
	\node at (2,0) (f){};
	\node[delt] at (1,3) (d){};
	\node[mult] at (2,1) (m){};
	\draw[dpoint] (c) to (d);
	\draw[dpoint] (d)[out=225,in=90] to (b);
	\draw[common] (d)[out=-45,in=45] to (m);
	\draw[common] (m) to (f);
	\draw[common] (a)[out=270,in=135] to (m);
}
\newcommand{\qdef}{
    \node at (0,3) (a){};
    \node at (2,3)(b){};
    \node at (1,0)(c){};
    \node at (3,0)(f){};
    \node [mult] at (1,1) (m){};
    \node [delt] at (2,2) (d){};
    \draw[defect](a)[out=270,in=135] to (m);
    \draw[dpoint](m) to (c);
    \draw[dpoint](b) to (d);
    \draw[dpoint](d) to (m);
    \draw[common](d)[out=-45,in=90] to (f);
}
\newcommand{\qred}{
    \node at (0,3) (a){};
    \node at (2,3)(b){};
    \node at (1,0)(c){};
    \node at (3,0)(f){};
    \node [mult] at (1,1) (m){};
    \node [delt] at (2,2) (d){};
    \draw[dpoint](a)[out=270,in=135] to (m);
    \draw[dpoint](m) to (c);
    \draw[common](b) to (d);
    \draw[common](d) to (m);
    \draw[common](d)[out=-45,in=90] to (f);
}
\newcommand{\rred}{
	\node at (0,4) (c){};
	\node at (2,4) (a){};
	\node at (3,0) (f){};
	\node at (1,0) (b){};
	\node[delt] at (2,3) (d){};
	\node[mult] at (1,1) (m){};
	\draw[common] (a) to (d);
	\draw[common] (d)[out=225,in=90] to (f);
	\draw[common] (d)[out=-45,in=45] to (m);
	\draw[dpoint] (m) to (b);
	\draw[dpoint] (c)[out=270,in=135] to (m);
}

\newcommand{\sred}{
    \node at (1,6) (c){};
    \node at (0,4) (b){};
    \node at (0,1) (a){};
    \node[mult] at (2,5) (upm){};
	\node[mult] at (1,3) (lom){};
	\node[delt] at (1,2) (d){};
	\draw[dpoint] (b) to (lom);
	\draw[dpoint] (lom) to (d);
	\draw[dpoint] (d) to (a);
	\draw[common] (c) to (upm);
	\draw[common] (upm) to (2,4) to (lom);
	\draw[common](d) to (2.75,1)[out=-45,in=0] to (2.75,.5)[out=180,in=225] to (2.75,1.5)
	[out=45,in=45] to (upm);
}
\newcommand{\sdef}{
    \node at (1,6) (c){};
    \node at (0,4) (b){};
    \node at (0,1) (a){};
    \node[mult] at (2,5) (upm){};
	\node[mult] at (1,3) (lom){};
	\node[delt] at (1,2) (d){};
	\draw[defect] (b) to (lom);
	\draw[dpoint] (lom) to (d);
	\draw[dpoint] (d) to (a);
	\draw[dpoint] (c) to (upm);
	\draw[dpoint] (upm) to (2,4) to (lom);
	
	\draw[common](d) to (2.75,1)[out=-45,in=0] to (2.75,.5)[out=180,in=225] to (2.75,1.5)
	[out=45,in=45] to (upm);
}
\newcommand{\tdef}{
    \node at (0,5) (b){};
    \node at (0,2) (a){};
    \node at (1,0) (c){};
    \node [mult] at (1,4) (m){};
    \node [delt] at (1,3) (upd){};
    \node [delt] at (2,2) (lod){};
    \draw[dpoint](b) to (m);
    \draw[dpoint](m) to (upd);
    \draw[defect](upd) to (a);
    \draw[dpoint](upd) to (lod);
    \draw[dpoint](lod) to (c);
	
	\draw[common](lod) to (2.75,1)[out=-45,in=0] to (2.75,.5)[out=180,in=225] to (2.75,1.5)
	[out=45,in=45] to (m);
}
\newcommand{\tred}{
    \node at (0,5) (b){};
    \node at (0,2) (a){};
    \node at (1,.5) (c){};
    \node [mult] at (1,4) (m){};
    \node [delt] at (1,3) (upd){};
    \node [delt] at (2,2) (lod){};
    \draw[dpoint](b) to (m);
    \draw[dpoint](m) to (upd);
    \draw[dpoint](upd) to (a);
    
    \draw[common](upd) to (lod);
    \draw[common](lod) to (c);
	\draw[common](lod) to (2.75,1)[out=-45,in=0] to (2.75,.5)[out=180,in=225] to (2.75,1.5)
	[out=45,in=45] to (m);
}
\newcommand{\exdef}{
        \node [delt] at (1,5)(ud){};
        \node [delt] at (2,4)(ld){};
        \node [mult] at (2,2)(um){};
        \node [mult] at (1,1)(lm){};
        \node at (1,6)(f){};
        \node at (3,6)(b){};
        \node at (1,0)(a){};
        \node at (3,0)(c){};
        \draw[dpoint](f) to [out=270, in = 90](ud);
        \draw[dpoint](b) to [out=270,in=25](um);
        
        \draw[defect](ud) to [out=200, in = 120] (lm);
        \draw[dpoint](ud) to [out=-20,in=90](ld);
        
        \draw[defect](ld) to [out=200,in=120](um);
        \draw[dpoint](ld) to [out=-20,in = 90](c);
        
        \draw[dpoint](um) to [out=270,in=45](lm);
        
        \draw[dpoint](lm) to [out = 270,in=90](a);
}

\tikzset{
->-/.style={decoration={
  markings,
  mark=at position #1 with {\arrow{>}}},postaction={decorate}},
  delt/.style={draw,shape=circle},
  mult/.style={draw,shape=circle},
  defect/.style = {decoration={markings,mark=at position .5 with {\arrow{>}}},postaction={decorate},blue,very thick, dashed},
  dpoint/.style = {decoration={markings,mark=at position .5 with {\arrow{>}}},postaction={decorate},red,thin, dotted},
  common/.style = {decoration={markings,mark=at position .5 with {\arrow{>}}},postaction={decorate},black,very thin}
  }

\theoremstyle{plain}	
\newtheorem{theorem}{Theorem}[section]
\newtheorem{proposition}[theorem]{Proposition}
\newtheorem{conjecture}[theorem]{Conjecture}

\theoremstyle{definition} 
\newtheorem{definition}[theorem]{Definition}
\newtheorem{example}[theorem]{Example}

\begin{document}

\title[State-Sum Construction of TQFTs with Defects]{A General State-Sum Construction of 2-Dimensional Topological Quantum Field Theories with Defects}

\author{Gathoni Kamau-Devers}
\address{University of California, Berkeley, CA 94720, USA}

\author{Gail Jardine}
\address{Brigham Young University, Provo, UT 84602, USA}

 \author{David Yetter}
\address{Kansas State University, Department of Mathematics, Cardwell Hall, Manhattan, KS 66506, USA}
\email{dyetter@math.ksu.edu}

\begin{abstract}

We derive the general state sum construction for 2D topological quantum field theories (TQFTs) with source defects on oriented curves, extending the state-sum construction from special symmetric Frobenius algebra for 2-D TQFTs without defects (cf. Lauda \& Pfeiffer \cite{LP}). From the extended Pachner moves (Crane \& Yetter \cite{CY}), we derive equations that we subsequently translate into string diagrams so that we can easily observe their properties. As in Dougherty, Park and Yetter \cite{DPY}, we require that triangulations be flag-like, meaning that each simplex of the triangulation is either disjoint from the defect curve, or intersects it in a closed face, and that the extended Pachner moves preserve flag-likeness.  

This research was conducted under the mentorship of Prof. David Yetter at Kansas State University with the support of NSF grant DMS-1262877.
\end{abstract}

\maketitle

\section{Introduction}

The present paper generalizes both the construction of
Dougherty, Park and Yetter \cite{DPY} in which a state sum construction for finite gauge-group Dijkgraff-Witten theory on surfaces with co-dimension 1 defects was given and the restriction to an ordinary (closed) 2-dimensional TQFT of the construction given in \cite{LP} of open-closed 2-dimensional TQFTs from a state-sum using a symmetric Frobenius algebra as initial data.  

\section{An Oriented Curve on an Oriented Surface}

We start with an oriented curve on an oriented surface. By convention, we draw the orientation of the surface as counterclockwise. 

We then triangulate the curve and extend this to a triangulation of the surface satisfying

\begin{definition}
A triangulation of a curve-surface pair $\Gamma \subset \Sigma$ is {\em flag-like} if the intersection of every simplex of the triangulation with $C$ is either empty or a face of the simplex.
\end{definition}

The name is derived from general setting of stratified spaces, in which a triangulation is flag-like if and only if the restriction of the filtration of the space to any simplex is a simplicial flag (cf. \cite{Friedman}).

We observe that three types of edges arise in flag-like triangulations: those that do not intersect the curve, those that intersect the curve at exactly one point, and finally, those that are entirely on the curve. Thoughout our exposition we will denote edges not intersecting the curve with by solid black lines and label edges intersecting the curve at a point by red dotted lines. Finally, we depict vertices on the curve as large blue dots and edges on the curve as dashed blue lines.  Similarly there are three distinct types of triangles: those that contain a defect on a single vertex, those whose edge is the defect and those that are not interacting with the defect at all shown in this order below:

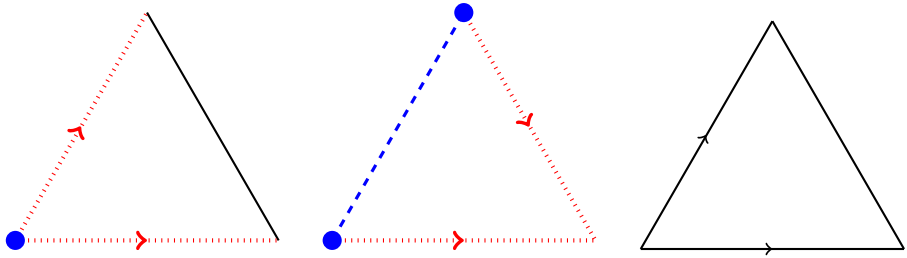
\vspace{8mm}
\begin{tikzpicture}[scale=3.5]

    \coordinate (0) at (0,0);
    \coordinate (2) at (1,0);
    \coordinate (1) at (0.5,0.866);
    \coordinate (3) at (0.5, 0.33);
    
    \draw[dpoint][ultra thick] (0) -- (1);
    \draw[thick] (1) -- (2);
    \draw[dpoint][ultra thick] (0) -- (2);

\draw node[fill=blue,circle,scale=0.7]{} (0,0);
    
\end{tikzpicture}
\hspace{3mm}
\nolinebreak
\begin{tikzpicture}[scale=3.5]

    \coordinate (0) at (0,0);
    \coordinate (2) at (1,0);
    \coordinate (1) at (0.5,0.866);
    
    \draw[blue][dashed][very thick] (0) -- (1);
    \draw[dpoint][ultra thick] (1) -- (2);
    \draw[dpoint][ultra thick] (0) -- (2);
    
    \draw node[fill=blue,circle,scale=0.7] at (0,0){};
    \draw node[fill=blue,circle,scale=0.7] at (0.5,0.866){};\hfill
\end{tikzpicture}
\hspace{3mm}
\nolinebreak
\begin{tikzpicture}[scale=3.5]

    \coordinate (0) at (0,0);
    \coordinate (2) at (1,0);
    \coordinate (1) at (0.5,0.866);
    
    \draw[common][thick] (0) -- (1);
    \draw[thick] (1) -- (2);
    \draw[common][thick] (0) -- (2);\hfill
\end{tikzpicture}
\vspace{8mm}

In our construction, the edges of the curve will be oriented according to the orientation of the curve, those incident with the curve at a vertex will be oriented away from the curve, and those not incident with the curve will be oriented according to an arbitrary ordering of the vertices not on the curve from the earlier vertex to the later.  In the figure above orientations are indicated on only two of the edges -- those oriented away from the curve, or those oriented away from the earliest numbered vertex.  Both orientations of the remaining edge can occur.

For each triangle, regardless of the type, either two edges agree with the boundary orientation induced by the orientation of the surface, while the third edge disagrees, or one agrees with the boundary orientation, while two disagree.  

If $\mathcal T$ is a flag-like triangulation, we denote the set of $n$-simplexes of $\mathcal T$ by ${\mathcal T}_n$, the set of $n$-simplexes of with $k$ vertices lying on the curve by ${\mathcal T}_n^k$, and for $n = 2$, the set of $2$-simplexes with $k$ vertices lying on the curve and two (resp. one) edge oriented in agreement with the boundary orientation by ${\mathcal T}_n^{k +}$ (resp. ${\mathcal T}_n^{k -}$).

\section{The State-Sum}

Fix a field $K$.  Usually we will have in mind the field of complex numbers ${\mathbb C}$, but the construction will work for an arbitrary field.  

Our state-sum construction begins with three sets of labels, $A$, $B$, and $C$, with local states given by coloring each edge not incident with the curve (resp. incident with the curve at a vertex, lying on the curve) with an element of $A$ (resp. $B$, $C$).

The local value of a state at a triangle is then given by the value of one of six functions, $a:A^3 \rightarrow K$, $\bar{a}:A^3 \rightarrow K$, $b:B \times A \times B \rightarrow K$, $\bar{b}:B^2 \times A \rightarrow K$, $c:C \times B^2 \rightarrow K$, and $\bar{c}:B \times C \times B \rightarrow K$, which we call {\em coefficient functions}.

In the construction, the arguments of $a$ (resp $\bar{a}$, $b$, $\bar{b}$, $c$, $\bar{c}$) will be the labels on the edges of triangles in ${\mathcal T}_2^{0'+}$ (resp. ${\mathcal T}_2^{0,-}$ , ${\mathcal T}_2^{1,+}$ , ${\mathcal T}_2^{1,-}$ , ${\mathcal T}_2^{2,+}$ , ${\mathcal T}_2^{2,-}$). For functions denoted by an unbarred letter the first two arguments, written as subscripts, will be the labels on the two edges oriented in agreement with the boundary orientation, in the order they occur according to the orientation, and the third argument, written as a superscript, will be the label on the edge oriented opposite to the boundary orientation. For functions denoted with a barred letter the first argument, written as a subscript, will be the label on the edge oriented in agreement with the boundary orientation, and the second and third arguments, written as superscripts, will be the labels of the remaining edges in the order they occur when traversed according to their orientation (as oriented edges).

\begin{definition}
A {\em $A,B,C$- coloring} (or simply {\em coloring}) of an a flag-like triangulation $\mathcal T$ of a curve-surface pair $\Gamma \subset \Sigma$ equipped with an ordering of ${\mathcal T}^0_0$ is a function 

\[ \lambda:{\mathcal T}_1\rightarrow A \coprod B \coprod C \]

such that $\epsilon \in {\mathcal T}_1^0$ (resp. $\epsilon \in {\mathcal T}_1^1$, $\epsilon \in {\mathcal T}_1^2$) implies $\lambda(\epsilon) \in A$ (resp. $\lambda(\epsilon) \in B$, $\lambda(\epsilon) \in C$).
\end{definition}

Although in specific instances of our local state functions, we write their arguments as sub- and super-scripts as discussed above, if $\lambda$ is a coloring of a flag-like triangulation $\mathcal T$, $\sigma \in {\mathcal T}_2$ and $f$ is the local state function applicable to its edge labels as discussed above, we will denote by $f(\lambda(\sigma))$ the value of the local state function on the edge labels assigned to $\sigma$ by $\lambda$ in the appropriate order.

We can now define a state-sum associated by the data $A, B, C, a, \bar{a}, b, \bar{b}, c, \bar{c}$ to triangulation $\mathcal T$ of a curve-surface pair $\Gamma \subset \Sigma$

\begin{eqnarray} \label{unnormalized}
Z(\Gamma \subset \Sigma, {\mathcal T}) & = & \sum_{\lambda}\left(
\prod_{\sigma\in T_2^{0,+}} a\left(\lambda\left(\sigma\right)\right)
\prod_{\sigma\in T_2^{0,-}} \bar{a}\left(\lambda\left(\sigma\right)\right)
\prod_{\sigma\in T_2^{0,+}} b\left(\lambda\left(\sigma\right)\right)\right. \nonumber \\
 & & \left. 
\prod_{\sigma\in T_2^{0,-}} \bar{b}\left(\lambda\left(\sigma\right)\right)
\prod_{\sigma\in T_2^{0,+}} c\left(\lambda\left(\sigma\right)\right)
\prod_{\sigma\in T_2^{2,-}} \bar{c}\left(\lambda\left(\sigma\right)\right)
\right),
\end{eqnarray}

\noindent which we call the {\em unnormalized state-sum}.  It will be our primary occupation in the balance of the paper to discover necessary and sufficient conditions on the data  $A, B, C, a, \bar{a}, b, \bar{b}, c, \bar{c}$ which ensure that a natural normalization by multiplying by quantities depending only on coarse combinatorial properties of the triangulation will be independent of the triangulation $\mathcal T$ and the ordering of ${\mathcal T}^0_0$.

\section{(Extended) Pachner Moves and Invariance Equations in Scalar Form}

As was shown in \cite{DPY}, two curve-surface pairs $\Gamma_i \subset \Sigma_i$, for $i = 1,2$, equipped with piece-wise linear structures are PL homeomorphic if and only if a flaglike triangulation of the one can be made combinatorially equivalent to a flaglike triangulation of the other by a sequence of extended Pachner (bistellar) moves in the sense of \cite{CY}, each of which preserves the flaglikeness.  For brevity we say such moves are {\em flaglike} (As was observed in \cite{DPY}, the barycentric subdivision of any triangulation of a surface, for which the curve is a subcomplex, is a flaglike triangulation, so the PL structures are determined by the flaglike triangulations.)

Below, we depict the six different flaglike Pachner move and one flaglike extended Pachner move for triangulations of curve-surface pairs. 

There are 35 scalar equations that follow from the seven flaglike moves on the triangulations shown below. As in the discussion above of triangles, we indicate orientation on edges only when they are required by our orientation conventions or may be assumed without loss of generality.  Each feasible assignment of orientations to the unoriented edges gives rise to an equation, which the coefficient functions must satisfy, or which must be restored by some normalization, if the state-sum is to give rise to a topological invariant of the curve-surface pair.

We give first the equations giving invariance under 2-2 Pachner moves, then those corresponding to invariance under the 1-3 Pachner moves, and finally that corresponding to invariance under the 2-4 extended Pachner move.  In what follows, we will see that the equations below arising from the 1-3 (resp. 2-4) moves can only be satisfied in trivial cases, and must be restored by introducing a normalization factor depending on the  number of vertices in the compliment of the curve (resp. on the curve).

\begin{center}
\begin{tikzpicture}

  \coordinate (0) at (0,0);
  \coordinate (1) at (75:3);
  \coordinate (2) at ($(1) +(3,0)$);
  \coordinate (3) at ($(2) +(-105:3)$);

  \draw[common][thick] (0) -- (1)node[font=\small,midway,left]{$r$};
  \draw[common][thick] (0) -- (3)node[font=\small,midway,below]{$q$};
  \draw[thick] (1) -- (2)node[font=\small,midway,above]{$s$};
  \draw[thick] (2) -- (3)node[font=\small,midway,right]{$t$};
  \draw[thick] (1) -- (3)node[font=\small,midway,below]{$u$};

\end{tikzpicture}
\nolinebreak
\eqarrow{

\coordinate (0) at (0,0);
\coordinate (1) at (75:3);
\coordinate (2) at ($(1) +(3,0)$);
\coordinate (3) at ($(2) +(-105:3)$);\hfill

\draw[common][thick] (0) -- (1) node[font=\small,midway,left]{$r$};
\draw[common][thick] (0) -- (3) node[font=\small,midway,below]{$q$};
\draw[common][thick] (0) -- (2) node[font=\small,midway,below]{$v$};
\draw[thick] (1) -- (2)node[font=\small,midway,above]{$s$};
\draw[thick] (2) -- (3)node[font=\small,midway,right]{$t$};

}
\end{center}

\begin{eqnarray}
\sum a_{qu}^r a_{ts}^u & = & \sum a_{vs}^r a_{qt}^v \\
\sum \bar{ a}_q^{ru}\bar{ a}_u^{st} & = & \sum \bar{ a}_v^{rs} \bar{ a}_q^{vt}\\
\sum a_{qu}^r \bar{ a}_t^{us} & = & \sum \bar{ a}_v^{rs} a_{qt}^v \\
\sum \bar{ a}_q^{ru} a_{ut}^s & = &\sum \bar{ a}_v^{rs} a_{qt}^v\\
\sum \bar{ a}_q^{ru} a_{su}^t & = &\sum a_{vs}^r \bar{ a}_q^{vt}\\
\sum a_{qu}^r \bar{ a}_s^{tu} & = & \sum a_{vs}^r \bar{ a}_q^{vt}
\end{eqnarray}
\medskip

\begin{center}
\begin{tikzpicture}

  \coordinate (0) at (0,0);
  \coordinate (1) at (75:3);
  \coordinate (2) at ($(1) +(3,0)$);
  \coordinate (3) at ($(2) +(-105:3)$);

  \draw[dpoint][ultra thick] (0) -- (1)node[font=\small,midway,left]{$r$};
  \draw[dpoint][ultra thick] (0) -- (3)node[font=\small,midway,below]{$q$};
  \draw[thick] (1) -- (2)node[font=\small,midway,above]{$s$};
  \draw[thick] (2) -- (3)node[font=\small,midway,right]{$t$};
  \draw[thick] (1) -- (3)node[font=\small,midway,below]{$u$};
  
  \draw node[fill=blue,circle,scale=0.7]{} (0,0);

\end{tikzpicture}
\nolinebreak
\eqarrow{

\coordinate (0) at (0,0);
\coordinate (1) at (75:3);
\coordinate (2) at ($(1) +(3,0)$);
\coordinate (3) at ($(2) +(-105:3)$);\hfill

\draw[dpoint][ultra thick] (0) -- (1) node[font=\small,midway,left]{$r$};
\draw[dpoint][ultra thick] (0) -- (3) node[font=\small,midway,below]{$q$};
\draw[dpoint][ultra thick] (0) -- (2) node[font=\small,midway,below]{$v$};
\draw[thick] (1) -- (2)node[font=\small,midway,above]{$s$};
\draw[thick] (2) -- (3)node[font=\small,midway,right]{$t$};
\draw node[fill=blue,circle,scale=0.7]{} (0,0);
}
\end{center}

\begin{eqnarray}
\sum b_{qu}^r a_{ts}^u & = & \sum b_{vs}^r b_{qt}^v\\
\sum \bar{ b}_q^{ru}\bar{ a}_u^{st} & = & \sum \bar{ b}_v^{rs} \bar{ b}_q^{vt}\\
\sum b_{qu}^r \bar{ a}_t^{us} & = & \sum \bar{ b}_v^{rs} b_{qt}^v\\
\sum \bar{ b}_q^{ru} a_{ut}^s & = & \sum \bar{ b}_v^{rs} a_{qt}^v\\
\sum \bar{ b}_q^{ru} a_{su}^t & = & \sum b_{vs}^r \bar{ b}_q^{vt}\\
\sum b_{qu}^r \bar{ a}_s^{tu} & = & \sum b_{vs}^r \bar{ b}_q^{vt}
\end{eqnarray}
\medskip

\begin{center}
\begin{tikzpicture}

  \coordinate (0) at (0,0);
  \coordinate (1) at (75:3);
  \coordinate (2) at ($(1) +(3,0)$);
  \coordinate (3) at ($(2) +(-105:3)$);

  \draw[thick][blue][dashed] (0) -- (1) node[font=\small,midway,left]{$r$};
  \draw[dpoint][ultra thick] (0) -- (3) node[font=\small,midway,below]{$q$};
  \draw[dpoint][ultra thick] (1) -- (2) node[font=\small,midway,above]{$s$};
  \draw[thick] (2) -- (3) node[font=\small,midway,right]{$t$};
  \draw[dpoint][ultra thick] (1) -- (3) node[font=\small,midway,below]{$u$};
  
  \draw node[fill=blue,circle,scale=0.7]{} (0,0);
  \draw node[fill=blue,circle,scale=0.7]{};
  \draw node[fill=blue,circle,scale=0.7] at (1){};
  
\end{tikzpicture}
\nolinebreak
\eqarrow{

\coordinate (0) at (0,0);
\coordinate (1) at (75:3);
\coordinate (2) at ($(1) +(3,0)$);
\coordinate (3) at ($(2) +(-105:3)$);\hfill

\draw[blue][thick][dashed] (0) -- (1) node[font=\small,midway,left]{$r$};
\draw[dpoint][ultra thick] (0) -- (3) node[font=\small,midway,below]{$q$};
\draw[dpoint][ultra thick] (0) -- (2) node[font=\small,midway,below]{$v$};
\draw[dpoint][ultra thick] (1) -- (2) node[font=\small,midway,above]{$s$};
\draw[thick] (2) -- (3) node[font=\small,midway,right]{$t$};
\draw node[fill=blue,circle,scale=0.7]{} (0,0);
\draw node[fill=blue,circle,scale=0.7] at (1){};
}
\end{center}

\begin{eqnarray}
\sum c_{rq}^u b_{ut}^s & = & \sum c_{rv}^s b_{qt}^v\\
\sum \bar{ c}_q^{ru} \bar{ b}_u^{st} & = & \sum \bar{ c}_v^{rs} \bar{ b}_q^{vt}\\
\sum \bar{ c}_q^{rv} b_{vt}^s & = & \sum \bar{ c}_u^{rs} b_{qt}^u\\
\sum c_{rq}^u \bar{ b}_u^{st} & = & \sum c_{rv}^s \bar{ b}_q^{vt}
\end{eqnarray}

\bigskip

Even without the expectation based on Lauda and Pfeiffer's state-sum construction of open-closed 2-dimensional TQFTs \cite{LP} that our initial data should somehow involve an algebra, and indeed a Frobenius algebra, we would be quickly drawn to that conclusion by examining the equations above:  the equation involving only instances of $a$ (resp. only instances of $\bar{a}$) is easily recognized as the equations expressing the associativity (resp. coassociativity) of an algebra (resp. coalgebra) in terms of structure coefficients with respect to a basis.

Similarly the equation involving only instances of $a$ and $b$ (resp. $\bar{a}$ and $\bar{b}$) is plainly the expression in terms of structure coefficients for the axiom characterizing an action of an associative algebra (resp. coassociative coalgebra) with basis $A$ on a module (resp. comodule) with basis $B$.  

This suggests regarding the sets of colors, $A$, $B$, and $C$ as bases for vector spaces, which we will, by abuse of notation, denote by the same letter as their basis, because in what follows they will displace the original sets of colors, and coefficients $a$ (resp. $\bar{a}$, $b$, $\bar{b}$, $c$, $\bar{c}$) as the structure coefficients for an associative multiplication
\[ \mu:A\otimes A \rightarrow A\]
(resp. a coassociative comultiplication
\[ \delta:A\rightarrow A\otimes A,\]
a right action
\[ \mu_B:B\otimes A \rightarrow B,\]
a right coaction
\[ \delta_B:B \rightarrow B\otimes A,\]
a left ``action''
\[ {\sf m}:C\otimes B\rightarrow B,\]
and a left ``coaction''
\[ {\sf d}:B\rightarrow C\otimes B ).\]

In the last two instances the use of the words ``action'' and ``coaction'' is a bit cavalier:  Even when we have worked out the complete structure, we will see that $C$ need not have any internal structure of the sort which usually justifies the use of the words action or coaction, and will use scare quotes to remind the reader of this fact. 

For instance $\mu$ is given by $s\otimes t \mapsto \sum_u a_{st}^u u$ whenever $s$ and $t$ are in the basis for $A$, and the sum runs over all basis elements $u$, while the ``coaction'' of $C$ on $B$ is given by $q \mapsto \sum_{ru}  \bar{c}_q^{ru} r\otimes u$ whenever $q$ is in the basis for $B$, $r$ ranges over the basis for $C$ and $u$ ranges over the basis for $B$.

The equations arising from the 2-2 Pachner moves will be seen to be mostly familiar equational axioms on these maps.  The equations arising from the 1-3 Pachner move and 2-4 extended Pachner move given below are of a less clear import.

\bigskip
\begin{center}
\begin{tikzpicture}[scale=3.5]

    \coordinate (0) at (0,0);
    \coordinate (2) at (1,0);
    \coordinate (1) at (0.5,0.866);
    
    \draw[common][thick] (0) -- (1) node[font=\small,midway,left]{$r$};
    \draw[thick] (1) -- (2) node[font=\small,midway,right]{$s$};
    \draw[common][thick] (0) -- (2) node[font=\small,midway,below]{$t$};
    
  \node at (1.25,0.43)(center){$\leftrightarrow$}; 
\end{tikzpicture}
\nolinebreak
\begin{tikzpicture}[scale=3.5]

    \coordinate (0) at (0,0);
    \coordinate (2) at (1,0);
    \coordinate (1) at (0.5,0.866);
    \coordinate (3) at (0.5, 0.33);
    
    \draw[common][thick] (0) -- (1) node[font=\small,midway,left]{$r$};
    \draw[thick] (1) -- (2) node[font=\small,midway,right]{$s$};
    \draw[common][thick] (0) -- (2) node[font=\small,midway,below]{$t$};
    
  \draw[thick] (0) -- (3)node[font=\small,midway,below]{$\sigma$};
  \draw[thick] (1) -- (3)node[font=\small,midway,right]{$\tau$};
  \draw[thick] (3) -- (2)node[font=\small,midway,below]{$\rho$};\hfill  
\end{tikzpicture}
\end{center}


\begin{eqnarray}a_{ts}^r & = & \sum_{\rho,\sigma,\tau}\bar{ a}_\sigma^{r\tau} a_{s\tau}^\rho a_{t\rho}^\sigma\\
a_{ts}^r & = & \sum_{\rho,\sigma,\tau}a_{\sigma\tau}^r a_{\rho,s}^\tau \bar{ a}_t^{\sigma\rho}\\
a_{rs}^r & = & \sum_{\rho,\sigma,\tau} a_{t \rho}^\sigma \bar{ a}_s^{\rho \tau} a_{\sigma \tau}^r\\
a_{ts}^r & = & \sum_{\rho,\sigma,\tau}a_{\sigma t}^\rho a_{\rho s}^\tau \bar{ a}_\tau^{\sigma r}\\
\bar{ a}_t^{rs} & = & \sum_{\rho,\sigma,\tau} \bar{ a}_\sigma^{r\tau} \bar{ a}_\tau^{s\rho} a_{t\rho}^\sigma\\
\bar{ a}_t^{rs} & = & \sum_{\rho,\sigma,\tau}\bar{ a}_\tau^{\sigma r} \bar{ a}_\rho^{\tau,s} a_{\sigma t}^\rho\\
\bar{ a}_t^{rs} & = & \sum_{\rho,\sigma,\tau}a_{\sigma\tau}^r \bar{ a}_\rho^{\tau s}\bar{ a}_t^{\sigma\rho}\\
\bar{ a}_t^{rs} & = & \sum_{\rho,\sigma,\tau}\bar{ a}_\sigma^{r\tau} a_{\tau\rho}^s \bar{ a}_t^{\sigma\rho}
\end{eqnarray}
\medskip

\begin{center}
\begin{tikzpicture}[scale=3.5]

    \coordinate (0) at (0,0);
    \coordinate (2) at (1,0);
    \coordinate (1) at (0.5,0.866);
    \coordinate (3) at (0.5, 0.33);
    
    \draw[dpoint][ultra thick] (0) -- (1)node[font=\small,midway,left]{$r$};
    \draw[thick] (1) -- (2) node[font=\small,midway,right]{$s$};
    \draw[dpoint][ultra thick] (0) -- (2)node[font=\small,midway,below]{$t$};

\node at (1.25,0.43)(center){$\leftrightarrow$};  
\draw node[fill=blue,circle,scale=0.7]{} (0,0);
    
\end{tikzpicture}
\nolinebreak
\begin{tikzpicture}[scale=3.5]
    \coordinate (0) at (0,0);
    \coordinate (2) at (1,0);
    \coordinate (1) at (0.5,0.866);
    
    \draw[dpoint][ultra thick] (0) -- (1)node[font=\small,midway,left]{$r$};
    \draw[thick] (1) -- (2) node[font=\small,midway,right]{$s$};
    \draw[dpoint][ultra thick] (0) -- (2)node[font=\small,midway,below]{$t$};
    
  \draw node[fill=blue,circle,scale=0.7]{} (0,0);
  \draw[dpoint][ultra thick] (0) -- (3)node[font=\small,midway,below]{$\sigma$};
  \draw[thick] (1) -- (3) node[font=\small,midway,right]{$\tau$};
  \draw[thick] (3) -- (2) node[font=\small,midway,below]{$\rho$};\hfill  
\end{tikzpicture}
\end{center}

\begin{eqnarray}
\bar{ b}_t^{rs} & = & \sum_{\rho,\sigma,\tau}\bar{ b}_\sigma^{r\tau}\bar{ a}_\tau^{s\rho} b_{t\rho}^\sigma\\
\bar{ b}_t^{rs} & = & \sum_{\rho,\sigma,\tau} b_{\sigma\tau}^r \bar{ a}_\rho^{\tau s}\bar{ b}_t^{\sigma\rho}\\
\bar{ b}_t^{rs} & = & \sum_{\rho,\sigma,\tau}\bar{ b}_\sigma^{r \tau} a_{\tau \rho}^s \bar{ b}_t^{\sigma \rho}\\
b_{ts}^r & = & \sum_{\rho,\sigma,\tau}\bar{ b}_\sigma^{r\tau} a_{s\tau}^\rho b_{t\rho}^\sigma\\
b_{ts}^r & = & \sum_{\rho,\sigma,\tau} b_{\sigma\tau}^r a_{\rho s}^\tau \bar{ b}_t^{\sigma\rho}\\
b_{ts}^r & = & \sum_{\rho,\sigma,\tau} b_{t\rho}^\sigma \bar{a}_s^{\rho \tau} b^r_{\sigma \tau}
\end{eqnarray}

\medskip

\begin{center}
\begin{tikzpicture}[scale=3.5]

    \coordinate (0) at (0,0);
    \coordinate (2) at (1,0);
    \coordinate (1) at (0.5,0.866);
    
    \draw[blue][dashed][very thick] (0) -- (1) node[font=\small,midway,left]{$r$};
    \draw[dpoint][ultra thick] (1) -- (2) node[font=\small,midway,right]{$s$};
    \draw[dpoint][ultra thick] (0) -- (2) node[font=\small,midway,below]{$t$};
    
  \draw node[fill=blue,circle,scale=0.7] at (0,0){};
  \draw node[fill=blue,circle,scale=0.7] at (0.5,0.866){};
  \node at (1.25,0.43)(center){$\leftrightarrow$}; 
\end{tikzpicture}
\nolinebreak
\begin{tikzpicture}[scale=3.5]

    \coordinate (0) at (0,0);
    \coordinate (2) at (1,0);
    \coordinate (1) at (0.5,0.866);
    \coordinate (3) at (0.5, 0.33);
    
    \draw[blue][dashed][very thick] (0) -- (1)node[font=\small,midway,left]{$r$};
    \draw[dpoint][ultra thick] (1) -- (2) node[font=\small,midway,right]{$s$};
    \draw[dpoint][ultra thick] (0) -- (2) node[font=\small,midway,below]{$t$};
    
  \draw node[fill=blue,circle,scale=0.7] at (0,0){};
  \draw node[fill=blue,circle,scale=0.7] at (0.5,0.866){};
  \draw[dpoint][ultra thick] (0) -- (3) node[font=\small,midway,below]{$\sigma$};
  \draw[dpoint][ultra thick] (1) -- (3) node[font=\small,midway,right]{$\tau$};
  \draw[thick] (3) -- (2) node[font=\small,midway,below]{$\rho$};\hfill  
\end{tikzpicture}
\end{center}

\begin{eqnarray}
\bar{ c}_t^{rs} &  = & \sum_{\rho,\sigma,\tau}\bar{ c}_\sigma^{r\tau}\bar{ b}_\tau^{s\rho} b_{t\rho}^\sigma\\
c_{rt}^s & = & \sum_{\rho,\sigma,\tau} c_{r\sigma}^{\tau} \bar{b}_\tau^{s\rho} b_{t\rho}^\sigma\\
\bar{ c}_t^{rs} & = & \sum_{\rho,\sigma,\tau}\bar{ c}_\sigma^{r\tau} b_{\tau\rho}^s \bar{ b}_t^{\sigma\rho}\\
c_{rt}^s & = &\sum_{\rho,\sigma,\tau}c_{r\sigma}^\tau b_{\tau\rho}^s \bar{ b}_t^{\sigma\rho}
\end{eqnarray}
\medskip

\begin{center}
\begin{tikzpicture}

\coordinate (0) at (0,0);
\coordinate (1) at (75:3);
\coordinate (2) at ($(1) +(3,0)$);
\coordinate (3) at ($(2) +(-105:3)$);

\draw[dpoint][ultra thick] (1) -- (0) node[font=\small,midway,left]{$r$};
\draw[dpoint][ultra thick] (3) -- (0) node[font=\small,midway,below]{$q$};
\draw[dpoint][ultra thick] (1) -- (2) node[font=\small,midway,above]{$s$};
\draw[dpoint][ultra thick] (3) -- (2) node[font=\small,midway,right]{$t$};
\draw[thick][defect] (3) -- (1) node[font=\small,midway,below]{$f$};
\draw node[fill=blue,circle,scale=0.7] at (1){};
\draw node[fill=blue,circle,scale=0.7] at (3){};

\end{tikzpicture}
\nolinebreak
\eqarrow{

\coordinate (0) at (0,0);
\coordinate (1) at (75:3);
\coordinate (2) at ($(1) +(3,0)$);
\coordinate (3) at ($(2) +(-105:3)$);
\coordinate (4) at ($(0) + (1.87,1.5)$);

\draw[dpoint][ultra thick] (4) -- (0) node[font=\small,midway,above]{$g$};
\draw[dpoint][ultra thick] (4) -- (2) node[font=\small,midway,above]{$j$};
\draw[dpoint][ultra thick] (1) -- (0) node[font=\small,midway,left]{$r$};
\draw[dpoint][ultra thick] (3) -- (0) node[font=\small,midway,below]{$q$};
\draw[dpoint][ultra thick] (1) -- (2) node[font=\small,midway,above]{$s$};
\draw[dpoint][ultra thick] (3) -- (2) node[font=\small,midway,right]{$t$};
\draw[defect][thick] (4) -- (1) node[font=\small,midway,above]{$k$};
\draw[defect][thick] (3) -- (4) node[font=\small,midway,above]{$h$};
\draw node[fill=blue,circle,scale=0.7] at (1){};
\draw node[fill=blue,circle,scale=0.7] at (3){};
\draw node[fill=blue,circle,scale=0.7] at (1){};
\draw node[fill=blue,circle,scale=0.7] at (4){};\hfill
}
\end{center}

\begin{eqnarray}
\sum c_{fr}^q \bar{ c}_t^{fs} & = & \sum c_{kr}^g \bar{ c}_j^{ks} \bar{ c}_t^{hj} c_{hg}^q
\end{eqnarray}


\bigskip

\section{String Diagrams}

In preference to rewriting the 35 equations needed for topological invariance as equations between composites of monoidal prolongations of the maps $\mu, \delta, \mu_B, \delta_B, m,$ $d$, and the structure maps of $K{\rm -v.s.}_{fd}$ as a compact closed category, written as equations between symbol strings, we prefer to use string diagram notation (cf. \cite{JS, Y2}) to denote the vector forms of the equations.  

We will denote the vector spaces $A$, $B$ and $C$ by solid, dotted and dashed arcs, respectively, each oriented downward, their dual spaces by the same sort of arc oriented upward, and the operations by trivalent nodes:

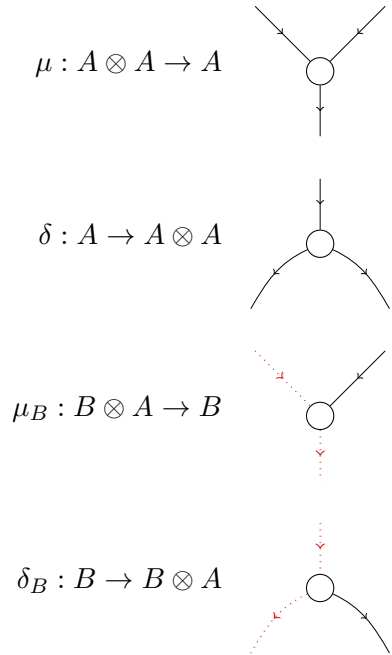
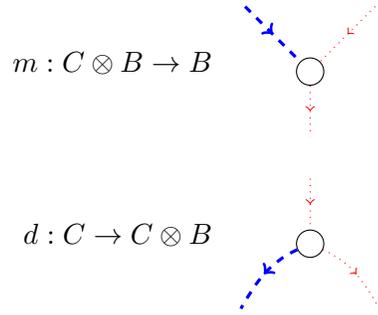
\begin{itemize}
\item[$\mu:A\otimes A \rightarrow A$]
\begin{center}\raisebox{-.5\height}{\begin{tikzpicture}\multi \end{tikzpicture}}\end{center}
\item[$\delta:A\rightarrow A\otimes A$]
\begin{center}\raisebox{-.5\height}{\begin{tikzpicture}\edelt \end{tikzpicture}}\end{center}
\item[$\mu_B:B\otimes A \rightarrow B$]
\begin{center}\raisebox{-.5\height}{\begin{tikzpicture}\actred \end{tikzpicture}}\end{center}
\item[$\delta_B:B\rightarrow B\otimes A$]
\begin{center}\raisebox{-.5\height}{\begin{tikzpicture}\eredelt \end{tikzpicture}}\end{center}
\item[$m:C\otimes B \rightarrow B$]
\begin{center}\raisebox{-.5\height}{\begin{tikzpicture}\actdef \end{tikzpicture}}\end{center}
\item[$d:C\rightarrow C\otimes B$]
\begin{center}\raisebox{-.5\height}{\begin{tikzpicture}\edefdelt \end{tikzpicture}}\end{center}

\end{itemize}

Note that there is no need to label the nodes, since the inputs and outputs identify which map is represented by the node.

As usual the symmetry is denoted by the arcs for the objects crossing, the evaluation map as a local minimum with the upward oriented arc on the left and the downward on the right, and the projective-coordinate or coevaluation map as a local maximum with the downward oriented arc on the left and the upward on the right.

The equations corresponding to the 2-2 Pachner moves and the 2-4 extended Pachner move can all be expressed using only the operations and the symmetry, while those corresponding to the 1-3 Pachner move will require use of the evaluation and coevaluation maps as well.  Below, we give the string diagram equations in the corresponding order to the scalar equations of the previous section, equation $(n+35)$ below being equivalent to equation $(n)$ above for $n = 2,\ldots, 36$.

\begin{eqnarray}
\begin{tikzpicture}\icom\end{tikzpicture} & = & \begin{tikzpicture}\jcom\end{tikzpicture} \\
\begin{tikzpicture}\kcom\end{tikzpicture} & = & \begin{tikzpicture}\lcom\end{tikzpicture} 
\end{eqnarray}

\begin{eqnarray}
\begin{tikzpicture}\qcom\end{tikzpicture} & = & \begin{tikzpicture}\ncom\end{tikzpicture} \\
\begin{tikzpicture}\mcom\end{tikzpicture} & = & \begin{tikzpicture}\ncom\end{tikzpicture} \\
\begin{tikzpicture}\ocom\end{tikzpicture} & = & \begin{tikzpicture}\pcom\end{tikzpicture} \\
\begin{tikzpicture}\rcom\end{tikzpicture} & = & \begin{tikzpicture}\ocom\end{tikzpicture} \\
\begin{tikzpicture}\ired\end{tikzpicture} & = & \begin{tikzpicture}\jred\end{tikzpicture} 
\end{eqnarray}

\begin{eqnarray}
\begin{tikzpicture}\kred\end{tikzpicture} & = & \begin{tikzpicture}\lred\end{tikzpicture} \\
\begin{tikzpicture}\qred\end{tikzpicture} & = & \begin{tikzpicture}\nred\end{tikzpicture} \\
\begin{tikzpicture}\mred\end{tikzpicture} & = & \begin{tikzpicture}\nred\end{tikzpicture} \\
\begin{tikzpicture}\ored\end{tikzpicture} & = & \begin{tikzpicture}\preda\end{tikzpicture} \\
\begin{tikzpicture}\rred\end{tikzpicture} & = & \begin{tikzpicture}\ored\end{tikzpicture} \\
\begin{tikzpicture}\idef\end{tikzpicture} & = & \begin{tikzpicture}\jdef\end{tikzpicture} \\
\begin{tikzpicture}\kdef\end{tikzpicture} & = & \begin{tikzpicture}\ldef\end{tikzpicture} 
\end{eqnarray}

\begin{eqnarray}
\begin{tikzpicture}\mdef\end{tikzpicture} & = & \begin{tikzpicture}\ndefa\end{tikzpicture} \\
\begin{tikzpicture}\ndefb\end{tikzpicture} & = & \begin{tikzpicture}\qdef\end{tikzpicture} \\
\begin{tikzpicture}\multi\end{tikzpicture} & = & \begin{tikzpicture}\scom\end{tikzpicture} \\
\begin{tikzpicture}\multi\end{tikzpicture} & = & \begin{tikzpicture}\ccom\end{tikzpicture} \\
\begin{tikzpicture}\multi\end{tikzpicture} & = & \begin{tikzpicture}\dcom\end{tikzpicture} 
\end{eqnarray}

\begin{eqnarray}
\begin{tikzpicture}\multi\end{tikzpicture} & = & \begin{tikzpicture}\bcom\end{tikzpicture} \\
\begin{tikzpicture}\edelt\end{tikzpicture} & = & \begin{tikzpicture}\tcom\end{tikzpicture}  \\
\begin{tikzpicture}\edelt\end{tikzpicture} & = & \begin{tikzpicture}\fcom\end{tikzpicture} \\
\begin{tikzpicture}\edelt\end{tikzpicture} & = & \begin{tikzpicture}\gcom\end{tikzpicture} 
\end{eqnarray}

\begin{eqnarray}
\begin{tikzpicture}\edelt\end{tikzpicture} & = & \begin{tikzpicture}\hcom\end{tikzpicture} \\
\begin{tikzpicture}\eredelt\end{tikzpicture} & = & \begin{tikzpicture}\tred\end{tikzpicture} \\
\begin{tikzpicture}\eredelt\end{tikzpicture} & = & \begin{tikzpicture}\gred\end{tikzpicture} \\
\begin{tikzpicture}\eredelt\end{tikzpicture} & = & \begin{tikzpicture}\hred\end{tikzpicture} 
\end{eqnarray}

\begin{eqnarray}
\begin{tikzpicture}\actred\end{tikzpicture} & = & \begin{tikzpicture}\sred\end{tikzpicture} \\
\begin{tikzpicture}\actred\end{tikzpicture} & = & \begin{tikzpicture}\cred\end{tikzpicture} \\
\begin{tikzpicture}\actred\end{tikzpicture} & = & \begin{tikzpicture}\dred\end{tikzpicture} \\
\begin{tikzpicture}\edefdelt\end{tikzpicture} & = & \begin{tikzpicture}\tdef\end{tikzpicture}
\end{eqnarray}

\begin{eqnarray}
\begin{tikzpicture}\actdef\end{tikzpicture} & = & \begin{tikzpicture}\sdef\end{tikzpicture} \\
\begin{tikzpicture}\edefdelt\end{tikzpicture} & = & \begin{tikzpicture}\hdef\end{tikzpicture} \\
\begin{tikzpicture}\actdef\end{tikzpicture} & = & \begin{tikzpicture}\ddef\end{tikzpicture} \\
\begin{tikzpicture}\odef\end{tikzpicture} & = & \begin{tikzpicture}\exdef\end{tikzpicture}
\end{eqnarray}

The reader familiar with string diagram representations of algebraic structures in monoidal categories will immediately recognize equations (37) to (40) as those axioms of a Frobenius algebra in the multiplication comultiplication form that involve neither the unit nor the counit.  As observed in \cite{AH} results of Kaplansky \cite{Ka} and Koch \cite{Ko} imply that these axioms together with the existence of a counit for the comultiplication imply that the existence of a unit for the multiplication is equivalent to the underlying vector space being finite dimensional.  Dually the existence of a unit for the multiplication implies the existence of a counit.  

It is an open question at this writing whether there are any examples of finite dimensional models for equations (37) through (40) which are neither unital nor counital.  After some unsuccessful effort to find such an example we make

\begin{conjecture} \label{onlyFrob}
If $A, \mu, \delta$ is a finite dimensional vector space equipped with an associative multiplication $\mu$ and a coassociative comultiplication $\delta$, moreover satisfying the equations given in string notation by equations (39) and (40), then $A$ admits a (co)unit for its (co)multiplication, and is, in fact, a Frobenius algebra.
\end{conjecture}

In the balance of this paper, we will assume the existence of a unit element and counit functional, thus assuming that $A$ is, in fact, a finite dimensional Frobenius algebra.

Equations (43) to (46) and (48) to (52) have a similar from, except that in the former the left-most strand denotes a tensorand of $B$, rather then $A$, while in the latter the middle tensorand is $B$ and the leftmost is $C$. Equations (43) (resp. (44)) is the statement that $B$ has an action (resp. coaction) of $A$ as an algebra (resp. coalgebra), while equations (48) (resp. (49)) are the statement that the  left ``action'' of $C$ and the right action of $A$ commute (resp. the left ``coaction'' of $C$ and the right coaction of $A$ commute). It is natural to make

\begin{definition} For a Frobenius algebra $A$ a {\em right Frobenius module} over $A$ is a vector space equipped with an action $\mu_B:B\otimes A \rightarrow B$ and a coaction $\delta_B:B \rightarrow A\otimes B$ moreover satisfying $\sum (ba)_{[1]}\otimes (ba)_{(2)} = \sum b\cdot a_{(1)}\otimes a_{(2)}$ and
$\sum (ba)_{[1]}\otimes (ba)_{(2)} = \sum b_{[1]}\otimes b_{(2)}a$ (equations (45) and (46) written in Sweedler notation).
\end{definition}

and

\begin{definition} A left ``coaction'' (resp. ``action'') of $C$ on $B$ is {\em Frobenius with respect to} a right action (resp. coaction) of $A$ on $B$ if the coaction and action (resp. action and coaction) satisfy the equation given in string notation by equation (52) (resp. (51)).
\end{definition}

Equations (54) and (55) (resp. (59) and (60)) are all equivalent in the presence of associativity (resp. coassociativity) and the hypothesis that the multiplication is surjective (resp. the comultiplication is injective), which hypotheses are immediate for a Frobenius algebra with unit and counit, are equivalent to the condition usually called {\em specialness}, that is that $m(\delta) = Id_A$.  

If $A$ is a special Frobenius algebra acting and coacting or the vector space $B$, equations (63) and (66) are redundant, and equations (62) and (65) become equivalent to the condition in the following definition, as do equations (69) and (70) in the presence of the ``action'' and ``coaction'' of $C$ commuting with the action and coaction of $A$ on $B$:

\begin{definition} A Frobenius module $B$ over a (special) Frobenius algebra $A$ is a {\em special Frobenius module} if it satisfies $\mu_B(\delta_B) = Id_B$.  
\end{definition}

To this point, we have considered all of the equations that involve neither the symmetry (twist) nor the evaluation or coevaluation, and have found (leaving aside the open question of whether examples that are neither unital nor counital exist and assuming the existence of a (co)unit) that our initial data must consist of a special Frobenius algebra $A$, a special Frobenius $A$-module, $B$, and a vector space $C$ with an ``action'' and ``coaction'' commuting with the action and coaction of $A$. 

It is well known that the ``Frobenius relations'' of equations (39) and (40) imply that the counit is the functional $\epsilon:A\rightarrow k$ in the classical definition of Frobenius algebra as a unital associative algebra equipped with a linear functional $\epsilon$ such that $(a,b)\mapsto \epsilon(ab)$ is a non-degenerate bilinear pairing.  It also follows from the same equations that if we use the pairing 
$(a,b)\mapsto \epsilon(ab)$ to identify $A^*$ with $A$, regarding the pairing as the evaluation map, the corresponding coevaluation (or projective coordinate system) map is given by $\delta(1)$.

In terms of bases for $A$, this may be seen as saying that there are two ordered bases for $A$, $\{e_1,\ldots e_n\}$ and $\{e^1,\ldots e^n\}$, such that $\epsilon(e^i e_j) = \delta_{ij}$ (where here $\delta$ is the Kronecker delta) and $\delta(1) = \sum_{i=1}^n e_i \otimes e^i$.

In string diagrams we represent the unit as a univalent node at the top of an arc and the counit as a univalent node at the bottom of an arc:

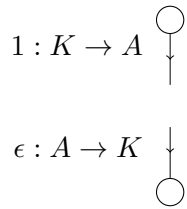
\begin{itemize}
\item[$1:K \rightarrow A$]
\begin{center}\raisebox{-.5\height}{\begin{tikzpicture}\unit \end{tikzpicture}}\end{center}
\item[$\epsilon:A\rightarrow K$]
\begin{center}\raisebox{-.5\height}{\begin{tikzpicture}\counit \end{tikzpicture}}\end{center}
\end{itemize}

The pairing and its corresponding coevaluation map are then denoted by

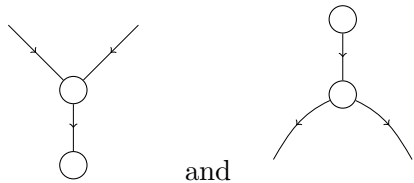
\begin{eqnarray*}
\begin{tikzpicture}\pairing\end{tikzpicture} & \mbox{\rm and} & \begin{tikzpicture}\copairing\end{tikzpicture}
\end{eqnarray*}

The existence of a counit (resp. unit) allows us to express the multiplication (resp. comultiplication) in terms of the comultiplication (resp. multiplication) and the bilinear paring (resp. corresponding coevaluation map) discussed above in two ways using relations (39) and (40), as shown for relation (39) the counit and multiplication here:

\begin{eqnarray}
\begin{tikzpicture}\qcomwcounit\end{tikzpicture} & = & \begin{tikzpicture}\ncomwcounit\end{tikzpicture}  = 
\begin{tikzpicture}\multi\end{tikzpicture}
\end{eqnarray}

\noindent the first equality by relation (39) the second by the defining property of a counit.  Similarly composing the right output of (40) with the counit gives another expression for the multiplication in terms of the comultiplication and the pairing, and precomposing the appropriate input with the unit give similar expressions for the comultiplication in terms of the multiplication and coevaluation map.

From this we can easily prove

\begin{proposition} \label{symmetricprop} If $A$ is a Frobenius algebra, the following are equivalent:

\begin{itemize}
\item relation (41) holds
\item relation (42) holds
\item $A$ is a symmetric Frobenius algebra, that is, the pairing $(a,b) \mapsto \epsilon(ab)$ is symmetric.
\item $A$ satisfies $\delta(1)$ = $\sigma_{A,A}(\delta(1))$, where $\sigma$ denotes the symmetry in the monoidal category of finite dimensional vector spaces over $K$.
\end{itemize}
\end{proposition}

\begin{proof}
That (41) or (42) implies that the Frobenius algebra is symmetric follows by composing both outputs with the counit, which immediately reduces the right-hand side of (41) and the left-hand side of (42) to the twisted pairing and reduces the other side of the equation to the pairing after passing the counit through the symmetry by naturality.

The converse implication for (41) follows by rewriting the multiplication on the left-hand side in the form derived from equation (40) (with string diagram the mirror image of that in the leftmost form in equation (72)) and applying coassociativity, and rewriting the multiplication on the right-hand side in the leftmost form of equation (72) then passing the comultiplication through the symmetry by naturality.  The resulting string diagrams are identical except that the latter has the pairing twisted.  Thus $A$ being symmetric implies (41) holds.  A similar argument shows the converse implication for (42).

Dual arguments show the equivalence of (41) and (42) with the last condition.
\end{proof}

The remaining equations involving only $A$, equations (53), (56), (57) and (58), readily follow from symmetry and specialness by using the identification of $A^*$ with $A$ given by using the Frobenius pairing and the coevaluation given by $\delta(1)$ as the evaluation and coevaluation maps.  For example, the reader can easily verify that once this substitution has been made, the right-hand side of (53) reduces to the multiplication by applying in sequence (and in the only way possible) associativity, the Frobenius relation (39), the defining property of the unit, symmetry, coassociativity, specialness and the defining property of the counit.

Having introduced the unit and counit, it is reasonable to insist that the action and coaction of $A$ on $B$ be unital and counital, in which case we have

\begin{proposition}
If $B$ is a Frobenius module over a symmetric Frobenius algebra $A$ with unital action and counital coaction, then equations (47), (48), (61) and (64) hold.  If, moreover $C$ has an an ``action'' and ``coaction'' on $B$ commuting with those of $A$, then (67) and (68) hold.
\end{proposition}

\begin{proof}
The proof of (47) and (48) is {\em mutatis mutandis} the same as that for the converses of Proposition \ref{symmetricprop}, while the proofs of (61) and (67) (resp. (64) and (68)) is {\em mutatis mutandis} that of (53) (resp. (57)) in the remark above.
\end{proof}

We have thus shown 

\begin{theorem}\label{precursor}

If $A$ is a special symmetric Frobenius algebra, $B$ a special (right) Frobenius module over $A$, and $C$  finite dimensional vector space equipped with a left ``action'' and ``coaction'' on $B$ commuting with the action and coaction of $A$, and moreover satisfying the equation given in string notion by (71), then using bases for $A$, $B$ and $C$ as colors and letting $a, \bar{a}, b, \bar{b}, c, \bar{c}$ be the structure coefficients for the multiplication, comultiplication, action of $A$ on $B$, coaction of $A$ on $B$, ``action'' of $C$ on $B$ and ``coaction'' of $C$ on $B$, respectively, the unnormalized state-sum associated to colorings of a triangulation $\mathcal T$ of a curve-surface pair $\Gamma \subset \Sigma$, $Z(\Gamma \subset \Sigma, {\mathcal T})$ is independent of the triangulation, and thus a topological invariant of the curve-surface pair.

\end{theorem}

The somewhat mysterious equation (71) simplifies if we posit more structure on $C$.  It is easy to see the following holds:

\begin{proposition}

If $C$ is a special Frobenius algebra, and its ``action'' and ``coaction'' on $B$, are in fact a left algebra action and a left coalgebra coaction, then equation (71) holds.

\end{proposition}

In practice, it is easier to find state-sum invariants by allowing a slightly less restrictive condition and normalizing the state-sum.  Observe that equations (53) to (70) all arise from Alexander subdivisions which introduce a new vertex off the curve, while (71) arises from an Alexander subdivision introducing a vertex on the curve -- in each case the string diagram in which there is a non-trivial cycle in the underlying graph (remembering that the string diagram for the symmetry should not be though of as a self-intersection) corresponding to the subdivided state with the extra vertex.

If we relax the definition of special Frobenius algebra to make

\begin{definition}
A Frobenius algebra $A, \mu, \delta, 1, \epsilon$ over $k$ is {\em projectively special} if there exists a non-zero
scalar $\rho \in k$ such that $\delta(\mu) = \rho Id_A$.

We call $\rho$ the {\em loop constant} of $A$.

A Frobenius module $B, \mu_B, \delta_B$ over a projectively special Frobenius algebra with loop constant $\rho$ is {\em projectively special} if the action and coaction satisfy $\delta_B(\mu_B) = \rho Id_B$.
\end{definition}

\noindent The discussion above proves our main result:

\begin{theorem}\label{main}

If $A$ is a projectively special symmetric Frobenius algebra with loop constant $\rho$, $B$ a projectively special (right) Frobenius module over $A$, and $C$ a projectively special Frobenius algebra with loop constant $\lambda$ with a left action and coaction on $B$ commuting with the action and coaction of $A$, then using bases for $A$, $B$ and $C$ as colors and letting $a, \bar{a}, b, \bar{b}, c, \bar{c}$ be the structure coefficients for the multiplication, comultiplication, action of $A$ on $B$, coaction of $A$ on $B$, action of $C$ on $B$ and coaction of $C$ on $B$, respectively, then the quantity

\[  \rho^{-|{\mathcal T}_0^0|} \lambda^{-|{\mathcal T}_0^1|} Z(\Gamma \subset \Sigma, {\mathcal T}) \]

\noindent is independent of the triangulation, and thus a topological invariant of the curve-surface pair.

\end{theorem}

Note that $B$ need not be projectively special as a left Frobenius module over $C$.

The previous theorem seems to give the construction at the ``right'' level of generality, in the sense that the initial data required is readily available.  It is, of course possible to state a more general result in which $C$ is still just a vector space with an ``action'' and ``coaction'' as in the previous theorem satisfying a projective version of equation (71), and a yet more general version of the construction if Conjecture \ref{onlyFrob} were to prove false.

It follows from generalities in \cite{Y} that the invariants thus obtained are the values on closed curve-surface pairs of a 2-dimensional topological quantum field theory with codimension one defects in the sense of Fuchs, et al. \cite{Fu}.

\section{Examples of Initial Data}

Our first two examples recapitulate the work of Dougherty, Park and Yetter \cite{DPY} in the context of our general setting

\begin{example}  Let $G$ and $H$ be finite groups, $X$ a set equipped with a left action of $H$ and a right action of $G$. Let $A = K[G]$, the group algebra of $G$, $C = K[H]$ and $B = Span(X)$,  the multiplications of the group algebras and their actions on $B$, together with comultiplications and coactions given on basis elements by sums over factorizations

\[ \Delta(x) = \sum_{yz = x} y\otimes z \]

\noindent for the comultiplication on $A$, the argument and the factors in the summands all lie in $G$, for the comultiplication on $C$ the argument and the factors in the summands all lie in $H$, for the coaction of $A$ on $B$, $x,y \in X$ and $z\in G$, and for the coaction of $C$ on $B$, $x,z \in X$ and $y \in H$, and the null infix denoted the group law or the action as appropriate, and counits for $A$ and $C$ given by the coefficient of the identity element.  Then $A$, $B$, $C$ and their (co)multiplications and (co)actions satisfy the hypotheses of Theorem \ref{main} with 
$A$ having loop constant $|G|$ and $B$ having loop constant $|H|$.

\end{example}

\begin{example}
Let $G$ and $H$ be finite groups, $X$ a set equipped with a left action of $H$ and a right action of $G$, and let $(\alpha, \beta, \gamma)$ be a generalized 2-cocycle in the sense of \cite{DPY} Theorem 4.1.  Then the vector spaces $A = Span(G)$, $B = Span(X)$, and $C = Span(H)$ with the operations given on basis elements by

\[ \mu(g\otimes f) = \alpha(g,f) gf \]

\[ \delta(g) = \sum_{\ell m = g} \alpha^{-1}(\ell, m) \ell\otimes m \]

\[ \mu_B(x \otimes g) = \beta(x,g) x\cdot g \]

\[ \delta_B(x) = \sum_{y\cdot g = x} \beta(y,g)^{-1} y\otimes g \]

\[ m(h\otimes x) = \gamma(h, x) h\cdot x \]

\[ d(x) = \sum_{h\cdot y = x} \gamma^{-1}(h,y) h\otimes y \]

\noindent satsify the hypotheses of Theorem \ref{main} with 
$A$ having loop constant $|G|$ and $C$ having loop constant $|H|$.

\end{example}

\begin{example} Let $A = Mat_{n\times n}(K)$, $B = Mat_{m\times n}(K)$ and $C = Mat_{m\times m}(K)$, with the multiplication $\mu$ on $A$, the right action $\mu_B$ of $A$ on $B$ and the left action $m$ of $C$ on $B$ all given by matrix multiplication, and the comultiplication $\delta$ on $A$, right coaction $\delta_B$ of $A$ on $B$, and the left coaction of $C$ on $B$ given on the bases of elementary matrices by

\[ E_{ij} \mapsto \sum_k E_{ik}\otimes E_{kj} ,\]

\noindent where the summand $k$ ranges from 1 to $n$ in the first two cases and from 1 to $m$ in the last and in each case the argument and the tensorands of the value are interpreted as lying in the appropriate vector spaces of matrices.   Then $A$, $B$, $C$ and their (co)multiplications and (co)actions satisfy the hypotheses of Theorem \ref{main} with 
$A$ having loop constant $n$ and $C$ having loop constant $m$.
\end{example}

Finally, a general class of examples in which some of the structure is trivial:

\begin{example}

Let $A$ be any projectively special symmetric Frobenius algebra with loop constant $\rho$, $B$ any projectively special Frobenius module over it (as, for example, $A$ with (co)action given by the (co)multiplication), and $C = k$ with scalar multiplication and its inverse as action and coaction on $B$. Then  $A$, $B$, $C$ and their (co)multiplications and (co)actions satisfy the hypotheses of Theorem \ref{main} with 
$A$ having loop constant $\rho$ and $C$ having loop constant $1$.

\end{example}

\end{document}